\documentclass[11pt]{amsart}

\usepackage{amsmath,amssymb,array,latexsym,amsthm,pdfsync,hyperref}
\usepackage{color}
\theoremstyle{definition}

\newtheorem{thm}{Theorem}
\newtheorem{lem}{Lemma}
\newtheorem{cor}{Corollary}
\newtheorem{prop}{Proposition}

\theoremstyle{remark}
\theoremstyle{remark}\newtheorem{rem}{Remark}
\theoremstyle{remark}\newtheorem{notat}{Notation}

\date{}
\subjclass[2010]{Primary 60F05; Secondary 60F17, 62E20}
\keywords{depth, consistency, central limit theorems, empirical processes}
\begin{document}

\title{Half-Region Depth for Stochastic Processes}
\author{James Kuelbs} 
\address{James Kuelbs\\Department of Mathematics,  University of Wisconsin, Madison, WI 53706-1388}
\email{kuelbs@math.wisc.edu}
\author{Joel Zinn}
\thanks{*Partially supported by NSF grant DMS-1208962.}
\address{\noindent Joel Zinn\\Department of Mathematics, Texas A\&M University, College Station, TX 77843-3368}
\email{jzinn@math.tamu.edu}

\begin{abstract} We study the concept of half-region depth, introduced by L\'opez-Pintado
and Romo in \cite{lp-r-half}. We show that for a wide variety of standard stochastic processes, such as Brownian motion and other symmetric stable processes with stationary independent increments tied down at 0, half-region depth assigns depth zero to all sample functions. To alleviate this difficulty we introduce a method of smoothing, which often not only eliminates the problem of zero depth, but allows us to extend
the theoretical results on consistency in that paper up to the $\sqrt n$ level for many smoothed processes.
\end{abstract}

\maketitle

\section{Introduction and Some Notation}\label{intro}
A number of depth functions are available to provide an ordering of finite dimensional data, and more recently in \cite{lp-r-half} the interesting notion of half-region depth for stochastic processes was introduced. This depth applies to data given in terms of infinite sequences, as functions defined on some interval, and even in more general settings. However, as we will see, one must exercise some care in its use.

In this paper we focus  on three issues. The first is to show (see section 2) that for many standard data sources this depth is identically zero, and hence the need to be cautious when employing it. A second issue we examine is how the problem of zero half-region depth can be avoided, and 
in Proposition \ref{smooth} it is shown that a smoothing of the data process will eliminate this problem. The third issue we consider involves limit theorems for the empirical half-region depth of these smoothed processes, and Theorem \ref{consist 1} is a basic consistency result with Theorem \ref{bdd prob clt} and Corollary \ref{rates} providing some rates of convergence for this consistency. In fact, it provides a sub-Gaussian tail bound. Now we turn to the notation used throughout the paper. In the remainder of this section we indicate some additional details as to how these issues are addressed, and how our results relate to other recent papers.

To fix some notation let $X:=\{X(t)=X_{t}\colon t\in T\}$ be a stochastic process on the probability space $(\Omega, \mathcal{F},P)$, all of whose sample paths are in $M(T)$, a linear space of real valued functions on $T$ which we assume to contain the constant functions. To handle measurability issues, we also {\bf always} assume that $ h \in M(T)$ implies
\begin{align}\label{Tcond}
 \sup_{t \in T}h(t) = \sup_{t \in T_0}h(t) < \infty, 
\end{align}
where $T_0$ is a fixed countable subset of $T$. Typical examples of $M(T)$ are the uniformly bounded continuous functions on $T$ when $T$ is a separable metric space, or the space of cadlag functions on $T$ for $T$ a compact interval of the real line. 
In either of these situations  $T_0$ could be any countable dense subset of $T$. It should also be observed that since (\ref{Tcond}) holds on the linear space $M(T)$, then $h \in M(T)$ implies
\begin{align}\label{Tcond2}
 \inf_{t \in T}h(t) = \inf_{t \in T_0}h(t) > - \infty~\rm {and}~ ||h||_{\infty} \equiv \sup_{t \in T}|h(t)| = \sup_{t \in T_0}|h(t)| < \infty.
\end{align}

If $g,h\colon T \rightarrow \mathbb{R}$ and $S \subseteq T$, let $g\preceq_S h$ (resp., $g\succeq_S h$),  denote that $g(t)\le h(t)$ (resp., $h(t)\ge h(t)$) for all $t\in S$. When $S=T$ we will simply write $g\preceq h$ (resp., $g\succeq h$). 
Then, for a function $h \in M(T)$,  the half-region depth with respect to $P$ is defined as 
\begin{align}\label{HR}
D(h,P):= D_{HR}(h,P):=\min(P(X \succeq h),P(X\preceq h)). 
\end{align}
To simplify, we also will write $D(h)$ for $D(h,P)$ when the probability measure $P$ is understood.
Since $M(T)$ is a linear space with (\ref{Tcond}) and (\ref{Tcond2}) holding, and the sample paths of the stochastic process $X$ are in $M(T)$, we see for each $h \in M(T)$ that 
\begin{align}\label{Tcond3}
\{ X \preceq h\}=\{ X \preceq_{T_0} h\}~{\rm {and}}~ \{X\succeq h\}= \{X\succeq_{T_0} h\}.
\end{align}
Thus the events in (\ref{HR}) are in $\mathcal{F}$ and the probabilities are defined.

Let $X_1,X_2, \cdots$ be i.i.d. copies of the process $X$, and assume $X,X_1,X_2,\cdots$ are defined on the probability space  $(\Omega,\mathcal{F},P)$ suitably enlarged, if necessary, and that all sample paths of each $X_j$ are in $M(T)$. Then, the empirical half-region
depth of $h \in M(T)$ based on the i.i.d. copies $X_1,\cdots,X_n$ is given by
\begin{align}\label{EHR}
D_{n}(h)=\min\{\frac{1}{n} \sum_{j=1}^n I( X_j \succeq h), \frac{1}{n} \sum_{j=1}^n I( X_j \preceq h)\}.
\end{align}

It is not surprising that like many other infinite dimensional problems, half-region depth is fraught with difficulties not found in the finite dimensional setting. In the next section we examine one such difficulty, namely, that there are classical situations in which the half-region depth is  equal to zero for all $h \in M(T)$. 
The recent paper \cite{dutta-tukey} obtains a result of similar type for Tukey's  half-space depth in the sequence space $\ell_2$, 
 and  here we'll present half-region depth examples that include data that appears as random sequences, and also as random functions from familiar continuous time stochastic processes. In particular, we will see sample continuous Brownian motion, tied down to be zero at $t=0$ with probability one, assigns zero half-region depth to all functions $h \in C[0,1]$, but this sort of problem also holds for other continuous time processes widely used to model data in a variety of settings.  Hence without suitable care, in these situations one is dealing with an object with little significance in the sense that if $X$ is a stochastic process with sample paths in M(T), and all $h \in M(T)$ have zero half-region depth with respect to $P=\mathcal{L}(X)$, then the implications for empirical consistency and central limit type behavior are trivial. That is, if the half-region depth function of every point $h \in M(T)$ is zero, then given $h$, either $I(X(t) \ge h(t)~ \forall~ t \in [0,1]) =
  0$ a.s.
   or 
$I(X(t) \le h(t)~ \forall~ t \in [0,1]) = 0$ a.s. with respect to $P.$ Since the empirical half-space depth function $D_n(h)$ given in (\ref{EHR}) is based on the minimum of two sums of such things, we have $D_n(h)- D(h) =D_n(h)=0$ a.s. with respect to $P$. Hence even the CLT is degenerate in this case.

Fortunately, in the final proposition of the next section we will see that 
in many situations smoothing the process by adding an independent real valued random variable $Z$  with a density as in (\ref{X Z smooth}) changes things dramatically 
for half-region depth. For example, sample continuous Brownian motion then attributes strictly positive depth to continuous functions.
In later sections we also present additional positive results for this depth. These include consistency results, and also some asymptotics at the $\sqrt n$ level. In these results the process $X$ will be as in (\ref{X Z smooth}),
and also satisfy some additional assumptions. 
It may also be worthwhile to mention that perhaps other forms of smoothing would be more suitable for other types of depth. This comment is motivated by the zero Tukey-depth result in \cite{dutta-tukey}, and also the zero projection depth results in \cite{chak-chaud-12}, which we found as we were in the final writing of this paper. 
Hence, it would be of interest to determine if a method of smoothing, of one sort or another, can be found to bypass this difficulty in other situations. 

In contrast to the smoothing we use, the paper  \cite{lp-r-half} also presents an alternative called modified half-region depth, which is non-degenerate at zero. There the depth itself is changed so as to be less restrictive, whereas here we retain the depth, but apply it to data which has been smoothed as in  (\ref{X Z smooth}). Moreover, the zero depth results we obtain are such that every function in the natural support of the process has depth zero, i.e. for sample continuous Brownian motion starting at zero at time zero, every continuous path has half-region depth zero. The results in  \cite{dutta-tukey} and  \cite{chak-chaud-12}  differ in that they show almost every function has zero depth with respect to the  the law of the process. Finally, we point out that the size of the collection of evaluation maps used in formulating a depth in the infinite dimensional setting, can make an enormous difference. If the collection is too large it is likely the depth will be degene
 rate, an
 d if it is too small the depth may not reveal details of importance in the data. This phenomenon also appears in connection with the central limit theorems we obtained for empirical processes and empirical quantile processes in \cite{kkz} and \cite{kz-quant}, where these CLTs may fail if the class of sets is too large, or there are degeneracies in the sample paths, as with Brownian motion tied down at zero. Again, smoothing helps, but one still needs to be careful, since the exact form of the depth and the evaluation maps used to define it can still produce unusual behavior. For example, in the setting of half-region depth the symmetric stable processes with stationary independent increments, cadlag paths on $[0,1]$, and tied down at $t=0$, are such that all cadlag paths on $[0,1]$ have half-region depth zero (Corollary \ref{ind incr} below), whereas by Proposition \ref{smooth} these processes smoothed as in (\ref{X Z smooth}) have positive depth. Moreover, they satisfy the 
consistency results and $\sqrt n$-asymptotics provided in Theorems 1,2, and 3. However, if we look at the increment half-region depth formed by differences of evaluations over only countably many disjoint subintervals of $[0,1]$ as in Corollary \ref{I hr d}, we see that both the smoothed and the unsmoothed version of these processes yield zero increment half-region depth for every function on $[0,1]$. Of course, similar comments apply to sample continuous Brownian motion, and we also have the half-region depth as defined in Corollary \ref{alpha  hr d1} degenerate at zero for all continuous functions on $[0,1]$ for both the smoothed and unsmoothed versions of Brownian motion.

\par
 \section{Zero Half-Region Depth and How It Can Be Eliminated}\label{zero depth}

The gist of this section is that  for many stochastic processes used in modeling data, half-region depth may be identically zero, but if we smooth the processes as indicated in  Proposition 4, this problem is eliminated. 

Subsection 2.1 deals with explicit classes of examples, and although these results demonstrate that zero half-region depth is a common phenomenon for many standard processes, the tools developed there should be useful when examining other processes for this problem. Furthermore, it should also be observed that the smoothing result in subsection 2.2, and the consistency and $\sqrt n$-asymptotics of sections 3 and 4, are independent of the proofs in subsection 2.1.

\subsection{Some Examples}

The half-region depths we examine first are for product probabilities $P$ on  the space of all real sequences $R(T)$, where $T= \{t:t =1,2\cdots\},$ and for each $h \in R(T)$  the half-region depth remains to be defined as in (\ref{HR}). As before we will write  $D(h)$ for $D(h,P)$ when the probability measure $P$ is understood.

For many such $P$ the uniformly bounded sequences $M(T)$ have probability zero, yet we still want to examine such situations as they are natural models of data sources, and they also  can be used (as in Corollaries \ref{alpha  hr d1} and \ref{I hr d}) to determine when  a half-region depth may be zero.
For example, if $P$ is the product probability whose coordinates are i.i.d. centered Gaussian
with variance one, then every coordinate-wise bounded sequence in $R(T)$ has half-region
depth equal to zero with respect to this $P$. Although the set of all such sequences has $P$-probability zero in this example, a little thought suggests much more may be true, and our next proposition shows that under rather broad circumstances the half-region depth may be zero for all sequences in $R(T).$ In particular, it applies to the Gaussian example we mentioned, and in Corollary 1 it also allows us to examine
the situation for sequences converging to zero, which are relevant when $P$ assigns mass one to  a  Banach sequence space such as $c_0$ or $\ell_p, 1\le p <\infty$.

Furthermore, if $P$ assigns probability one to $M(T)$, then using Proposition 4 at the end of this section we can show the half-region depth of every $h \in M(T)$ can be strictly positive for a smoothed version of the input data. This latter result applies to data indexed by countable or uncountable $T$, and $M(T)$ is as defined earlier. Of course, if $T$ is countably infinite, then $M(T)$ is a subset of the sequence space $\ell_{\infty}$, but our results also apply to many standard stochastic processes indexed by uncountable $T$.

Our first result provides necessary and sufficient conditions for half-region depth to be identically zero for $P$ a product measure on the sequence space $R(T).$ In contrast, a sufficient condition that implies a half-space depth is zero with $P$-probability one in $R(T)$ for various probabilities $P$, can be found in \cite{concerns}. However, these half-space depths are not zero everywhere, so determining when they are zero, when they are positive, and consistency issues for the related empirical depth are the main concerns there.

\begin{prop}\label{nasc} Let $\{Z_{t}:t \ge 1\}$ be independent rv's on the probability space $(\Omega, \mathcal{F},P)$ with  distribution functions $F_{t}$, and assume ${\bf a}=\{a_{t}\}_{t=1}^{\infty}$ is any sequence in $R(T)$. Then,
$$
D({\bf a},P)= 0
$$
if and only if  

\noindent (i) for at least one $t \in T,$  $P(Z_t \ge a_t)=0$ or $P(Z_t \le a_t)=0$, or
\bigskip

\noindent (ii) for all $t \in T,$ $P(Z_t \ge a_t)>0$ and $P(Z_t \le a_t)>0$, and
\begin{equation}\label{infinite sum}
\sum_{t \in T}P(Z_t \not= a_t)= \infty. 
\end{equation}
\end{prop}
\bigskip

\begin{rem}\label{first rem}  Under the conditions of Proposition \ref{nasc}, it is immediate that the the conclusion of Proposition \ref{nasc} is equivalent to the claim that
$$
D({\bf a},P)>0
$$
if and only if for all $t \in T$, $P(Z_t \ge a_t)>0$ and $P(Z_t \le a_t)>0$, and 
\begin{equation}\label{finite sum}
\sum_{t \in T} P(Z_t \not= a_t)<\infty.
\end{equation}

\end{rem}

\begin{proof} Under the assumptions of Proposition \ref{nasc}, it suffices 
 to prove Remark \ref{first rem}. To do this we first we note that 
\begin{align*}&D({\bf a}, P)\\
&=\min(P(Z_t \le a_t \text{ for all } t \ge 1),P(Z_t \ge a_t \text{ for all } t \ge 1))\\
&=\min(\prod_{t \ge 1} F_{t}(a_t), \prod_{t \ge 1}(1- F_{t}^{-}(a_t))),
\end{align*}
where $F_t^{-}(x)$ is the left limit at $x \in \mathbb{R}$. 

Hence, $D({\bf a},P)>0$ if and only if for all $t \in T$ we have $P(Z_t \ge a_t)>0$ and $P(Z_t \le a_t)>0$, and both the products
\begin{equation}\label{prod one}
\prod_{t \ge 1} F_{t}(a_t)= \prod_{t \ge 1}(1-P(Z_t>a_t)), 
\end{equation}
and
\begin{equation}\label{prod two}
 \prod_{t \ge 1}(1- F_{t}^{-}(a_t)))= \prod_{t \ge 1}(1- P(Z_t<a_t)) 
\end{equation}
are strictly positive. Since $P(Z_t \ge a_t)>0$ and $P(Z_t \le a_t)>0$
for all $t \in T$, the products in (\ref{prod one}) and (\ref{prod two}) are strictly positive if and only if
\begin{equation}\label{finite sum one}
\sum_{t \in T}P(Z_t>a_t)<\infty,  
\end{equation}
and
\begin{equation}\label{finite sum two}
\sum_{t \in T}P(Z_t<a_t)<\infty,  
\end{equation}
respectively. Now (\ref{finite sum one}) and (\ref{finite sum two}) holding is equivalent to (\ref{finite sum}), and hence the proof is complete.
\end{proof}

\begin{cor}\label{contin distrib}  Let $\{Z_{t}:t \ge 1\}$ be independent rv's on the probability space $(\Omega, \mathcal{F},P)$ with continuous distribution functions $F_{t}$ for $t \in T_1$, where $T_1$ is an infinite subset of $T$. Then,
$$
D({\bf a},P)= 0
$$
for ${\bf all}$ sequences ${\bf a}=\{a_{t}\}_{t=1}^{\infty}$ in  $R(T).$ 
Furthermore, if for  $t \in T_1$ and some $\delta>0$ we weaken the continuity assumption to $F_t$ being continuous on $(-\delta,\delta)$, then $D({\bf a},P)=0$ for ${\bf all}$ sequences ${\bf a}=\{a_{t}\}_{t=1}^{\infty}$ such that
$\lim_{t \rightarrow \infty} |a_t|=0$.

\end{cor}

\begin{proof}
If  the distribution functions $F_t$ are continuous on $\mathbb{R}$ for all $t \in T_1,$ where $T_1$ is an infinite subset of $T$, then $P(Z_t \not= a_t)=1$ for all such t's and (\ref{infinite sum}) holds. Thus
$P(Z_t \ge a_t)>0$ and $P(Z_t \le a_t)>0$ for all $t \in T$, and part (ii) of Proposition \ref{nasc}, implies $D({\bf a},P)=0$. Of course, if it is not the case that $P(Z_t \ge a_t)>0$ and $P(Z_t \le a_t)>0$ for all $t \in T$, then we also have $D({\bf a},P)=0$. 

If the assumption of continuity is weakened as indicated, then an entirely similar argument applies for all sequences converging to zero.
\end{proof}

\begin{rem}
In the previous corollary continuity of the distributions $F_t, t \in T,$ played an important role in showing zero half-region depth, but it clearly is not a necessary condition. For example, if $\{Z_t: t \in T\}$ are independent random variables with $P(Z_t=\pm c_t)=d_t, t \in T $, where $\{c_t: t \in T\}$ are strictly positive constants, $\sum_{t \in T}d_t = \infty$, and $F_t, t \in T,$ is arbitrary otherwise, then Proposition \ref{nasc} immediately implies 
for any sequence ${\bf a}=\{a_t: t \in T\}$ 
$$
D({\bf a},P)=0.
$$
\end{rem}

It is also easy to formulate two immediate consequences of Corollary \ref{contin distrib}, where natural sequential half-region depths will always be zero for probabilities which behave well in many instances, and are important in many modeling situations. Since more restrictions in the definition of a half-region depth make it easier for the depth to be zero, it is interesting to observe that in both examples the class of evaluation maps used to define the depths is again countably infinite. In  the first we assume $P$ is a centered Gaussian probability measure on a separable Banach space with infinite dimensional support. Then, it is well known that there are many sequences of continuous linear functionals $\mathcal{A} = \{\alpha_t: t \in T\} \subseteq B^{*}$  that  are i.i.d. centered Gaussian random variables with $\int_B\alpha_t^2(x)dP(x) =1$, and for $P$-almost all $x \in B$
$$
\lim_{n \rightarrow \infty}||x- \sum_{t=1}^n \alpha_t(x)S\alpha_t||=0,
$$
where $||\cdot||$ is the norm on $B$, and for each $ \alpha \in B^{*}$, $S \alpha$ is the Bochner integral $\int_B x \alpha (x)dP(x)$. Hence, with $P$-probability one the sequence $\mathcal{A}= \{\alpha_t: t \in T\}$ determines $x \in B$ in the sense that above series converges to $x$, and we define the $\mathcal{A}$-half-region depth of a vector ${\bf a} \in B$ to be
\begin{equation}\label{alpha hr d}
D_{\mathcal{A}}({\bf a},P)= \min \{P(\alpha_t(x) \ge \alpha_t({\bf a})~ \forall t \in T), P(\alpha_t(x) \le \alpha_t({\bf a})~ \forall t \in T) \}.
\end{equation}

\begin{cor} \label{alpha  hr d1}
If $P$ is a centered Gaussian measure on a separable Banach space with infinite dimensional support, and $\mathcal{A}=\{ \alpha_t: t \in T\} \subseteq B^{*}$ is as above, then for all ${\bf a} \in B$
\begin{equation}\label{alpha hr d2}
D_{\mathcal{A}}({\bf a},P)= 0.
\end{equation} 
\end{cor}

In the second application of Proposition \ref{nasc} we let $X=\{X(t): t \in [0,1]\}$ be a symmetric non-degenerate stable process with stationary independent increments and cadlag sample paths on $[0,1].$ If $X$ is tied down at $t=0$, then Proposition \ref{blumenthal} below shows that the half-region depth of every cadlag path on $[0,1]$ is zero with respect to $P$, and here we examine what might be considered a natural depth for the increments of these processes. Unfortunately, this depth is also zero for every function on $[0,1]$.

\begin{cor}\label{I hr d}
Let $\mathcal{I}= \{I_j=[u_j,v_j], j \ge 1\}$ consist of disjoint intervals of $[0,1]$, and define the increment half-region depth for every function $h$ on $[0,1]$ with respect to $P=\mathcal{L}(X)$ and $\mathcal{I}$  by
$$
D_{\mathcal{I}}(h,P)=\min \{P(X(I_j) \ge h(I_j) ~\forall~j\ge 1), P(X(I_j) \le h(I_j) ~\forall~j\ge 1) \},
$$
where $f(I_j)=f(v_j)-f(u_j)$ for every function $f $ on $[0,1]$. Then, 
\begin{equation}\label{I hr d2}
D_{\mathcal{I}}({\bf a},P)= 0.
\end{equation} 
\end{cor}

As mentioned above, both Corollaries \ref{alpha  hr d1} and \ref{I hr d} are immediate from Corollary \ref{contin distrib}, and the continuity of the relevant distribution functions.

The next proposition will allow us to obtain several more typical examples of ``zero half-region depth''. 

\begin{prop}\label{difference depth} Let  $\{X(t)\colon t\in T\}$ and $\{Y(t):t \in T\}$ be i.i.d  stochastic processes on $(\Omega, \mathcal{F},P)$, all of whose sample paths are in the linear space of functions $M(T)$. If $h \in M(T)$ and  
\begin{align}\label{X-Ycond}
P(X-Y\preceq_{S} 0)=0
\end{align}
for some subset $S$ of $T_0$, then 
$D(h,P)=0$.
\end{prop}

\begin{proof}
If the depth of $h \in M(T)$ is positive, then the product, \hfill\break
$P(h\preceq X)\cdot P(X\preceq h)$, is positive. So, since we always are assuming (\ref{Tcond}), (\ref{Tcond3}) and (\ref{X-Ycond}), we then have  
\begin{align}0<P&(h\preceq_{T_0} X)\cdot P(X\preceq_{T_0} h)=P(h\preceq_{T_0} X, Y\preceq_{T_0} h)\notag \\
&\le P(Y\preceq_{T_0} X)\le P(Y-X\preceq_{S} 0)=0.
\end{align}
\end{proof}

\begin{cor}\label{ind incr}Let $X$ be an independent increment process with paths in the Skorohod space $D[0,1]$ such that 
\begin{enumerate}
\item\label{cont incr} the increments have a continuous distribution, and
\item 
$P(X(0)=0)=1$.
\end{enumerate} If $h\in D[0,1]$, then $D(h,P)=0$.
\end{cor}

\begin{proof} Let $Z=X-Y$, where $X$ and $Y$ are defined on the probability space $(\Omega,\mathcal{F},P)$, $Y$ is an independent copy of $X$, and $X$ and $Y$ have sample paths in $D[0,1]$. Using Proposition \ref{difference depth}, with $T_0$ the rational numbers in $[0,1]$ and $S=\{\frac{1}{k}: k=1,2,\cdots\},$ we only have to check that $P(Z\preceq_S 0)=0$. We'll assume not. But, by the (right) continuity at $t=0$ and telescoping terms  we have
\begin{align}\label{incre}
Z(\frac{1}{k})=\lim_{ r \rightarrow \infty}[Z(\frac{1}{k})-Z(\frac{1}{r+1})] = \lim_{r \rightarrow \infty} \sum_{j=k}^r\Delta_j(Z)=\sum_{j=k}^{\infty}\Delta_j(Z),
\end{align}
where $\Delta_j(Z)=[Z(\frac{1}{j})-Z(\frac{1}{j+1})].$ Therefore, by our choice of $S$ and (\ref{incre})
\begin{align*}&0<P(Z\preceq_S 0)=P(\sum_{j=k}^{\infty}\Delta_j(Z)\le 0,\forall k\ge 1)\\
&\le P(\sum_{j=k}^{\infty}\Delta_j(Z) \le 0, \text{ eventually  in } k).
\end{align*}
This last event is in the tail $\sigma$-field of $\{Z(\frac1{j})-Z(\frac1{j+1})\colon j\ge 1\}$, so by Kolmogorov's zero-one law and the symmetry of $Z$ , we have
\[P(\sum_{j=k}^{\infty}\Delta_j(Z)\le 0, \text{ eventually  in } k)=  
 P(\sum_{j=k}^{\infty}\Delta_j(Z)\ge 0, \text{ eventually  in } k),\]  
and both probabilities are one.
Hence, by (\ref{incre}) we have  
$$
P(Z(\frac{1}{k})=0 \text{ eventually in }k)=1,
$$ 
and therefore $P(Z(1/k)-Z(1/k+1)=0 \text{ eventually in }k)=1$.
By the independence of the increments this last statement is equivalent to 
\[\sum_{k=1}^{\infty}P(Z(1/k)-Z(1/k+1)\neq 0)<\infty.
\]
Since each term is $1$, we have a contradiction.
\end{proof}

\begin{rem}\label{rem1} Let $X=\{X(t): t \in [0,1]\}$ be a symmetric stable process with parameter $r \in (0,2]$, and stationary independent increments with paths in $D[0,1]$. If we also have $P(X(0)=0)=1$, then the conclusion of Corollary \ref{ind incr} immediately holds. If $r=2$ and $X$ is Brownian motion with continuous sample paths, then the result also holds in that setting. However, if $X$ is a Poisson process with parameter $\lambda>0$, then the first condition of Corollary (\ref{ind incr}) does not hold. And, if $\xi$ has an exponential distribution with mean $\lambda$,  then 
\[P(X(t)\le 0 \text{ for all } t\in[0,1])=P(\xi>1)>0.
\]
Therefore, the half-space depth of the $0$ function is positive. Of course, the same conclusion is valid for compound Poisson processes starting at zero with probability one.
\end{rem}

\begin{cor}\label{brown sheet}Let $X=\{X(t): t=(t_1,t_2) \in [0,1]\times[0,1]\}$ be a centered Brownian sheet with covariance
\begin{align}\label{bro cov}
E(X(t_1,t_2)X(s_1,s_2))=\min\{s_1,t_1\} \min\{s_2,t_2\},
\end{align}
and continuous paths on $T=[0,1] \times[0,1]$. If $h$ is a continuous function on $T$ and $P$ is the law of $X$, then the half-region depth $D(h,P)=0.$
\end{cor}

\begin{proof} Let $Z=X-Y$, where $X$ and $Y$ are defined on the probability space $(\Omega,\mathcal{F},P)$, $Y$ is an independent copy of $X$, and $X$ and $Y$ have sample paths in $C(T)$. Let $T_0$ be the subset of $T$ consisting of points with both coordinates rational numbers in $[0,1]$ and let $S=\{t \in T_0: t=(t_1,t_1)\}$. Using Proposition \ref{difference depth}, we only have to check that $P(Z\preceq_S 0)=0$. We'll assume not. 
Then,
$$
0<P(Z(t) \le 0, ~{\rm {for~all}}~t \in S)=P(B(u)\le 0 ~{\rm {for~all}~} u \in [0,1]\cap Q),
$$
where $Q$ is the rational numbers and $B(u)=Z(u,u), u \in [0,1]$. Since $\{B(u): u \in [0,1]\}$ is a Brownian motion process with continuous sample paths and $P(B(0)=0=1$, we have
$$
P(B(u)\le 0 ~{\rm {for~all}~} u \in [0,1]\cap Q)=P(B(u)\preceq_{[0,1]} 0)=0,
$$
where the last equality follows from Remark \ref{rem1}. 
\end{proof}

The next result applies to many Markov processes with or without independent increments. 

\begin{prop}\label{blumenthal}Assume the stochastic process $X=\{X_{t}\colon 0\le t\le 1\}$ has sample paths in the Skorohod space $D[0,1]$ and it satisfies 
\begin{enumerate}
\item the Blumenthal zero-one law at $t=0$, i.e., for every $A\in\mathcal{F}_{0}^{+}:=\cap_{t>0}\mathcal{F}_{t}$ we have $P(A)=0 \text{ or } 1$, where $\mathcal{F}_t= \cup_{0\le s\le t}\sigma(X_s)$ and $\sigma(X_s)$ is the minimal sigma-field making $X_s$ measurable, and 
\item for every $t>0$, $X(t)$ has a continuous distribution function. 
\end{enumerate}
Then, the half-region depth $D(h,P)=0$ for every $h \in D[0,1]$. 

\end{prop}
\begin{proof}
If $D(h,P)>$, then
\begin{align}\label{>0}
P(X(\cdot) &\ge_{[0,1]} h(\cdot))>0 
\end{align}
and
\begin{align}\label{<0}
P(X(\cdot) &\le_{[0,1]} h(\cdot))>0.
\end{align}
For $n \ge 1$, let
\[
E_n=\{ X(\cdot) \ge_{[0,1/n]} h(\cdot)\},
\] 
and
\[
F_n=\{ X(\cdot) \le_{[0,1/n]} h(\cdot)\}.
\] 
Then, for every integer $k$
$$
E=\{X(t) \ge h(t) ~\rm {eventually~as}~ t \downarrow 0\} = \cup_{n \ge k}E_n, 
$$
and
$$
F=\{X(t) \le h(t) ~\rm {eventually~as}~ t \downarrow 0\} = \cup_{n \ge k}F_n. 
$$
This implies $E \in \mathcal{F}_{1/k}, F \in \mathcal{F}_{1/k}$ for all $k \geq 1$, and therefore $E,F \in \mathcal{F}_{0}^{+}=\cap_{k=1}^{\infty}\mathcal{F}_{1/k}$.
Now (\ref{>0}) implies $P(E)>0$ and (\ref{<0}) implies $P(F)>0$, so the Blumenthal zero-one law implies
$P(E)=P(F)=1$.
Since the events $E_n$ and $F_n$ increase in $n$, we  have that there exists a $k_0$ such that $n \ge k_0$ implies 
$P(E_n) > 3/4$ and $P(F_n) > 3/4$. Hence
$$
P(E_n \cap F_n) > 1/2 \text{ for all } n \ge k_0.
$$
Since $E_n \cap F_n=\{X(t)=h(t)~ \forall  ~t \in [0,1/n]\}$, this is a contradiction to the fact that $X(t)$ has a continuous distribution for all $t>0$. Thus the half-region depth of $h \in D[0,1]$ must be zero.
\end{proof}

\subsection{Eliminating Half-Region Zero Depth By Smoothing}

Although sample continuous Brownian motion, tied down to be zero at $t=0$ with probability one, assigns zero half-region depth to all functions $h \in C[0,1]$, by starting the process randomly with a density changes things dramatically. This follows immediately from the next proposition, and hence in order to be assured half-region depth is non-trivial, we use smoothing in the results that follow in subsequent sections. Moreover, the precise assumptions used for smoothing in these later results are also important in other parts of their proofs. The smoothed stochastic process $\{X=\{ X(t): t \in T\}$  will be such that 
\begin{align}\label{X Z smooth}
X(t) = Y(t) + Z, t \in T,
\end{align}
where $Z$ is a real valued random variable independent of the process $Y=\{Y(t): t \in T\}$, $Z$ has density $f_Z(\cdot)$  on $\mathbb{R}$, $Y$ has sample paths in the linear space $M(T)$, and we are assuming $M(T)$ is such that (\ref{Tcond}) holds. Of course, then (\ref{Tcond2}) also holds, and since we are assuming $M(T)$ contains the constant functions on $T$, $X$ also has its sample paths in $M(T).$

\begin{prop}\label{smooth} Let $X(t) = Y(t) + Z, t \in T,$ where $Y= \{Y(t): t \in T\}$ has sample paths in the linear space $M(T)$ satisfying $(\ref{Tcond})$ and $Z$ is independent of $Y$ with density $f_Z$. If $f_Z>0$ a.s. with respect to Lebesgue measure on $\mathbb{R}$ and $h \in M(T)$, then the
half-region depth  of $h$ determined by $\{X(t): t \in T\}$ is strictly positive.
 \end{prop}
\begin{proof} Let $h \in M(T)$. Then, 
\begin{align}\label{sup contr}
P(X \succeq h) = \int_{-\infty}^{\infty} P(Y(t) \ge h(t)- u~\forall~t \in T~| Z=u) f_Z(u)du.
\end{align}
Since (\ref{Tcond}) holds there exists an constant $c>0$ such that $P(||Y||_{\infty} \le c)>\frac{1}{2}$ and hence for $u>2c+ ||h||_{\infty}$ we have
\begin{align}\label{sup contr 1}
P(Y(t) \ge h(t)-u~\forall~t \in T~| Z=u)\ge P(Y(t) \ge -2c~\forall~t \in T)>\frac{1}{2}.
\end{align}
Since $f_{Z}>0$ a.s., by combining (\ref{sup contr}) and (\ref{sup contr 1}) we have 
$$
P(X \succeq h) \ge \int_{2c+||h||_{\infty}}^{\infty} \frac{1}{2}f_Z(u)du>0.
$$
Similarly,  $P(X \preceq h)>0$ for all $h \in M(T)$, and hence
$D(h,P)>0$ for all $h \in M(T)$.
\end{proof}

\begin{rem} If $P(||Y||_{\infty} \le c)>0$ for all $c>0$, then it is easy to see from the proof of the previous proposition that the half-region depth could be strictly positive for some $h \in M(T)$ without
the density being strictly positive on all of $\mathbb{R}$. 
\end{rem}

\begin{rem} Let $T=[a,b], -\infty<a<b <\infty$, and assume $M(T)$ denotes the real-valued cadlag paths on $T$. If $X=\{X(t): t \in T\}$ has paths in $M(T)$, and $Z:=X(a)$ is independent of $\{Y(t)=X(t)-X(a), t  \in T\}$ with density $f_Z>0$  a.s. with respect to Lebesgue measure on $\mathbb{R}$,
then Proposition \ref{smooth} implies the half-region depth with respect to $P=\mathcal{L}(X)$ is strictly positive on $M(T).$ Hence, under these conditions no smoothing is required to be certain the depth is strictly positive.
\end{rem}

\section{Consistency for Empirical Half-region Depth}

The consistency result we prove depends on  two lemmas, which are also important for the $\sqrt n$ asymptotics we obtain in Theorem \ref{bdd prob clt}.The proof of consistency is an application of empirical process ideas involving the Blum-Dehardt Theorem and bracketing entropy.

Let $X(t) = Y(t) + Z, t \in T,$ where $Y= \{Y(t): t \in T\}$ has sample paths in the linear space $M(T)$ satisfying $(\ref{Tcond}),$ and $Z$ is independent of $Y$ with density $f_Z$. Also, assume $X_1,X_2, \cdots$ are i.i.d. copies of the process $X$ with sample paths in $M(T)$ and that $X,X_1,X_2,\cdots$ are defined on the probability space  $(\Omega,\mathcal{S},P)$. Then, with half-region depth and half-region empirical depth defined as in (\ref{HR}) and (\ref{EHR}), and since for real numbers $a,b,c,d$ 
\begin{align}\label{min-min}
|\min\{ a,b\} -\min\{c,d\}| \le |a-c|+|b-d|,
\end{align}
the classical strong law of large numbers implies for each $h \in M(T)$
\begin{align}\label{classicalLLN}
\lim_{n \rightarrow \infty} |D_{n}(h)-D(h)|=0 
\end{align}
with probability one. 
The theorem below refines (\ref{classicalLLN}) to be uniform over $h \in E$, where $E$ is a suitably chosen subset of $M(T)$.
\bigskip

\begin{notat}For a function $f: \Omega \rightarrow \bar{\mathbb{R}}$ we use the notation $f^{*}$ to denote a measurable cover function (see Lemma 1.2.1 \cite{vw}). 
\end{notat}

\begin{thm}\label{consist 1}
Let $X(t) = Y(t) + Z, t \in T,$ where $Y= \{Y(t): t \in T\}$ has sample paths in the linear space $M(T)$, 
and $Z$ is independent of $Y$ with density $f_Z(\cdot)$  on $\mathbb{R}$ that is absolutely continuous and its derivative $f_Z^{'}(\cdot)$ is  in $L_1(\mathbb{R})$. Also, assume $X_1,X_2, \cdots$ are i.i.d. copies of the process $X$ with sample paths in $M(T)$ and that $X,X_1,X_2,\cdots$ are defined on the probability space  $(\Omega,\mathcal{F},P)$.
If $E$ is subset of $M(T)$
such that for every $r>0$
\begin{align}\label{compE}
E_{r} = E\cap \{f \in M(T): ||f||_{\infty} \le r\}
\end{align}
is a sup-norm compact subset of $M(T)$,
then with probability one
\begin{align}\label{LLN}
 \lim_{n \rightarrow \infty} \sup_{h \in E}|D_n(h)-D(h)|^{*}=0. 
\end{align}
\end{thm}
\bigskip

In order to prove this result we first establish some lemmas which will also be useful in our refinements  of (\ref{LLN}) that follow below. For $h \in M(T)$ we define
the stochastic process
$\{W_h: h \in M(T)\}$ on $(\Omega,\mathcal{F},P)$, where
\begin{align}\label{Wdef}
W_h \equiv W(h)= \inf_{t \in T} (X(t)-h(t)), h \in M(T). 
\end{align}
\bigskip

\begin{lem}\label{den}
Let $f$ be a probability density on $\mathbb{R}$ which is absolutely continuous and such that its derivative $f^{'}$ is in $L^{1}(\mathbb{R})$. 
Then,
\begin{align}\label{den 1}
\int_{\mathbb{R}} |f(x+\delta) -f(x)|dx \le |\delta| \int_{\mathbb{R}} |f^{'}(x)|dx. 
\end{align}
\end{lem}

\begin{proof} If $\delta \ge 0$, then
\begin{align} \int_{\mathbb{R}}&|f(x+\delta)-f(x)|\, dx=\int_{\mathbb{R}}|\int_{x}^{x+\delta}f^{'}(u)\, du|\, dx\notag\\
&\le \int_{\mathbb{R}}\int_{x}^{x+\delta}|f^{'}(u)|\, du \, dx=\int_{\mathbb{R}}\int_{\mathbb{R}}|f^{'}(u)|I_{x\le u\le x+\delta}\, du\, dx\notag\\
&=(\text{by Fubini}) \int_{\mathbb{R}}|f^{'}(u)|\int_{\mathbb{R}}I_{x\le u\le x+\delta}\, dx\, du=\delta\int_{\mathbb{R}}|f^{'}(u)|\,du,\notag
\end{align}
which gives (\ref{den 1})). The case $\delta<0$ follows similarly.
\end{proof}
\bigskip

\begin{lem}\label{den 2}
Let $X$ be as in (\ref{X Z smooth}) with $Y$ and $Z$ satisfying the assumptions of Theorem \ref{consist 1}, and assume $W_h$ be as in (\ref{Wdef}). Then,
for $ h_1, h_2 \in M(T)$ we have
\begin{align}\label{W 1}
|W_{h_1} - W_{h_2}|\le ||h_1 - h_2||_{\infty}.
\end{align}
Hence, if $||h_1- h_2||_{\infty} \le \delta$, then
\begin{align}\label{W 2}
|P(W_{h_{1}} \ge x) - P(W_{h_{2}} \ge x)| \le P(x- \delta \le W_{h_{1}}\le x) + P(x-\delta \le W_{h_{2}}\le x), 
\end{align}
and we also have
\begin{align}\label{W 3}
|P(W_{h_{1}} \ge x) - P(W_{h_{2}} \ge x)| \le 2 \delta \int_{\mathbb{R}} |f^{'}_Z(x)|dx. 
\end{align}
\end{lem}
\bigskip

\begin{proof} 
First observe that
\begin{align}\notag \inf_{t\in T}&\bigl(X(t) - h_{1}(t)\bigr)  \le  X(s) - h_{1}(s) 
= X(s) -h_{2}(s) + h_{2}(s) - h_{1}(s)\\
&\le X(s) - h_{2}(s) + \|h_{2}-h_{1}\|_{\infty} \text{ for all } s\in T.\notag
\end{align}
Hence,  
\begin{align}\notag W_{h_{1}}&=\inf_{t\in T}\bigl(X(t) - h_{1}(t)\bigr)  \le \inf_{s\in T}\bigl(X(s) - h_{2}(s)\bigr)+\|h_{2} - h_{1}\|_{\infty}\\
&=W_{h_{2}} +\|h_{2}-h_{1}\|_{\infty}.\notag
\end{align}
and interchanging $h_1$ and $h_2$ we have (\ref{W 1}).

Hence, if $||h_1 -h_2||_{\infty}\le \delta$, we then have from (\ref{W 1}) that
\begin{align}\label{W 4}
P(W_{h_1} \ge x) \le P(W_{h_2} \ge x) + P(x-\delta \le W_{h_2} \le x) 
\end{align}
and
\begin{align}\label{W 5} 
P(W_{h_2} \ge x) \le P(W_{h_1} \ge x) + P(x-\delta \le W_{h_1} \le x), 
\end{align}
and (\ref{W 4}) and (\ref{W 5}) combine to give (\ref{W 2}). 
\bigskip

To verify (\ref{W 3}) we define for $ h \in M(T)$ 
\begin{align}\label{W 6}
F(h,x)= P(W_{h}\ge x). 
\end{align}
>From (\ref{X Z smooth}), $F(h,x)= P(\inf_{t \in T}(Y(t)-h(t)) +Z \ge x)$,  and hence the independence of $Y$ and $Z$ implies
\begin{align}\label{W 7}
F(h,x)= \int_{\mathbb{R}}P_Y(\inf_{t \in T}(Y(t) -h(t)) \ge x-y)f_Z(y)dy. 
\end{align}
Letting $\xi(h)= \inf_{t \in T}(Y(t) -h(t))$, we see (\ref{W 7}) implies
\begin{align}\label{W 8}
F(h,x_1)- F(h,x_2)= \int_{\mathbb{R}}[P_Y(\xi(h) \ge x_1-y)- P_Y(\xi(h) \ge x_2-y)]f_Z(y)dy 
\end{align}
$$
=\int_{\mathbb{R}}[P_Y(\xi(h) \ge s)[f_Z(x_1-s) -f_Z(x_2-s)]ds.
$$
Therefore,
\begin{align}\label{W 9}
|F(h,x_1)- F(h,x_2)| \le\int_{\mathbb{R}}|f_Z(x_1-s) -f_Z(x_2-s)|ds,  
\end{align}
and setting $u=x_1-s$ we have $x_2-s=(x_2-x_1) +u$, so  Lemma \ref{den} implies
$$
|F(h,x_1)- F(h,x_2)| \le\int_{\mathbb{R}}|f_Z(u) -f_Z(u+(x_2-x_1))|du \le |x_1 -x_2|\int_{\mathbb{R}}|f_Z^{'}(x)|dx.  
$$
Thus the lemma is proven since (\ref{W 2}) and the above combine to give (\ref{W 3}) when $||h_1- h_2||_{\infty} \le \delta$.
\end{proof}
\bigskip

\begin{proof}
In order to verify (\ref{LLN}) we first will show 
for every $\epsilon>0$ there is an $r_0< \infty$ such that 
the strong law of large numbers and  
implies with probability one that
\begin{align}\label{W 14}
\limsup_{n \rightarrow \infty}\sup_{\{h: ||h||_{\infty} \ge r_0\}} D_n(h)\le P(||X||_{\infty} \ge r_0) \le \epsilon, 
\end{align}
and
\begin{align}\label{W 12}
\lim_{r \rightarrow \infty}\sup_{\{h: ||h||_{\infty} \ge r\}} D(h)=0. 
\end{align}
The argument  for (\ref{W 14}) and (\ref{W 12}) is essentially the proof of Proposition 5 in \cite{lp-r-half}, but the details are included below.

To prove (\ref{W 12}) we observe
$$
\sup_{||h||_{\infty} \ge r} D(h) \le A_r + B_r,
$$
where
$$
A_r= \sup_{||h||_{\infty} \ge r, ||h||_{\infty} =\sup_{t \in T}h(t) }  P( X \succeq h),
$$
and
$$
B_r= \sup_{||h||_{\infty} \ge r, ||h||_{\infty} =\sup_{t \in T}(-h(t)) }  P( X \preceq h).
$$
Thus  
$$
A_r \le  \sup_{||h||_{\infty}  =\sup_{t \in T}h(t) \ge r }  P( \sup_{t \in T} X(t) \ge \sup_{t \in T}h(t))~~~~~~~~~~~~~~~~~~~~~~~~~~
$$
$$
\le \sup_{||h||_{\infty} =\sup_{t \in T}h(t)\ge r }  P( ||X||_{\infty} \ge ||h||_{\infty}) \le P(||X||_{\infty} \ge r),
$$
and 
$$
B_r \le  \sup_{||h||_{\infty}  =\sup_{t \in T}(-h(t)) \ge r }  P( \inf_{t \in T} X(t) \le \inf_{t \in T}h(t))~~~~~~~~~~~~~~~~~~~~~~~~~~
$$
$$
\le \sup_{||h||_{\infty} =\sup_{t \in T}(-h(t))\ge r }  P( ||X||_{\infty} \ge ||h||_{\infty}) \le P(||X||_{\infty} \ge r),
$$
and hence we have  (\ref{W 12}). To prove (\ref{W 14}) we note that
$$
\sup_{||h||_{\infty} \ge r} D_n(h) \le A_{r,n} + B_{r,n}.
$$
where
$$
A_{r,n}= \sup_{||h||_{\infty} \ge r, ||h||_{\infty} =\sup_{t \in T}h(t) }  \frac{1}{n}\sum_{j=1}^n I( X_j \succeq h), 
$$
and
$$
B_{r,n}= \sup_{||h||_{\infty} \ge r, ||h||_{\infty} =\sup_{t \in T}(-h(t)) } \frac{1}{n}\sum_{j=1}^n I( X_j \preceq h).
$$
Thus, in similar fashion it follows that  
$$
A_{r,n} \le  \sup_{||h||_{\infty}  =\sup_{t \in T}h(t) \ge r }  \frac{1}{n}\sum_{j=1}^n I(||X_j||_{\infty} \ge ||h||_{\infty})\le \frac{1}{n}\sum_{j=1}^nI(||X_j||_{\infty} \ge r),
$$
and 
$$
B_{r,n} \le  \sup_{||h||_{\infty}  =\sup_{t \in T}(-h(t)) \ge r } \frac{1}{n}\sum_{j=1}^n I(||X_j||_{\infty} \ge ||h||_{\infty})\le \frac{1}{n}\sum_{j=1}^nI(||X_j||_{\infty} \ge r),
$$
and therefore we have  (\ref{W 14}).

Since $\epsilon>0$ is arbitrary, (\ref{W 12}) and (\ref{W 14}) combine to imply (\ref{LLN}) provided we show for every $r>0$ that with probability one
\begin{align}\label{W 15}
 \lim_{n \rightarrow \infty} \sup_{h \in E_{r}}|D_n(h)-D(h)|^{*}=0, 
\end{align}
where $E_r$ is defined as in (\ref{compE}). The proof of (\ref{W 15}) follows from the Blum-Dehardt Theorem using the  bracketing entropy for $E_r$ as  in \cite{Dudley-unif-clt}, p. 235. That is, since $E_r$ is compact in $M(T)$ with respect to the sup-norm, for every $\delta>0$ implies there exists finitely many  points $\{h_1,\cdots,h_{k(\delta)}\} \subseteq E_r$ such that
$$
\sup_{h \in E_{r}} \inf_{h_{j}}||h-h_j||_{\infty}\le \delta.
$$
In addition, the brackets $F(\delta,h_j)=\{ z \in M(T): h_j(t)-\delta \le z(t)\le h_j(t) + \delta\}$ have union covering $E_r$ with $ z \in F(\delta,h_j)$ implying
$$
I(X \succeq h_j +\delta) \le I(X\succeq  z) \le I(X\succeq h_j -\delta).
$$
Hence, for $\epsilon>0$ fixed, and $\delta= \delta(\epsilon)>0$ such that $4\delta \int_{\mathbb{R}}|f^{'}_Z(x)|dx \le \epsilon$, we have from (\ref{W 3}) that
$$
\mathcal{E}_r \equiv \{I( X\succeq h): h \in E_r\}
$$
is a subset of
$$
 \cup_{j=1}^{k(\delta(\epsilon))} \{I(X \succeq z): z \in E_r, I(X\succeq h_j+\delta) \le I(X \succeq z) \le I(X \succeq h_j-\delta)  \},
$$
and
$$
||I(X \succeq h_j -\delta)- I(X \succeq h_j+\delta) ||_{1}\le  \epsilon,
$$
where $||\cdot||_1$ denotes the $L_1$ norm with respect to $P$. Hence, for every $\epsilon>0$ we have $\mathcal{E}_r$ covered by finitely many  $L_1$-brackets of diameter $\epsilon$.  A similar argument can be made for 
$$
\mathcal{F}_r \equiv \{I( X \preceq h): h \in E_r\},
$$
and hence (\ref{W 15}) holds by (\ref{min-min})  and the Blum-Dehardt result mentioned above. 
Combining (\ref{W 12}), (\ref{W 14}), and (\ref{W 15}) we have (\ref{LLN}).
\end{proof}
\bigskip

\subsection{Some Remarks on the C1 Condition in \cite{lp-r-half}} 

Let $X= \{X_t: t \in [0,1]\}$ be a sample continuous stochastic process, and assume $P$ is the Borel probability on $C[0,1]$ induced by $X$.  The main focus of the paper \cite{lp-r-half} is the formulation of a consistency result for half-region depth that is uniform over an equicontinuous family of functions on $[0,1]$, where the depth is with respect to $X$, or equivalently the probability distribution $P$. One of the crucial assumptions in this endeavor is that $P$ satisfy their $C1$ condition, where 
\bigskip

\noindent C1: Given $\epsilon>0$, there exists a $\delta>0$, such that for every pair of functions  $h_1,h_2 \in C[0,1]$ with $||h_1 - h_2||_{\infty} \le \delta$ implies
\begin{align}\label{C 1}
P(h_1 \preceq_{[0,1]} X  \preceq_{[0,1]} h_2) <\epsilon. 
\end{align}
\bigskip

This condition appears on the bottom of page 1687 in \cite{lp-r-half}. The notation in  \cite{lp-r-half} is slightly different than that above, but (\ref{C 1}) is consistent with their use of C1 on page 1688 of \cite{lp-r-half}.
However, the main problem with (\ref{C 1}) as used in \cite{lp-r-half} is two-fold. First, in their proof of Theorem 3 of \cite{lp-r-half} it is applied to functions $h_1,h_2$ which are not continuous, and secondly it is claimed that for $h_1,h_2 \in C[0,1]$ with $h_1 \preceq h_2$
\begin{align}\label{C 1-2}
P(h_1\preceq X)-P(h_2 \preceq X)=P(h_1 \preceq X \preceq h_2),
\end{align}
which is far from being true since
\begin{align}
P(h_1\preceq X)-P(h_2 \preceq X)=P(\{h_1 \preceq X\}\cap\{h_2 \preceq X\}^{c}).
\end{align}

Hence there are some major concerns with their proof, and  in Theorem \ref{consist 1} above we obtained a result that alleviates such concerns. Moreover, we have taken care to discuss when half-region is non-trivial, and how to eliminate the problem of it being trivial by using smoothing. Another question one might ask is whether the quantity
\begin{align}\label{C 1-3}
|P(h_1\preceq X)-P(h_2 \preceq X)|,
\end{align}
can be made arbitrarily small when $h_1,h_2 \in C[0,1]$ provided $||h_1-h_2||_{\infty}$ is sufficiently small. This is an important ingredient in our proof, and we established sufficient conditions for the continuity posed for (\ref{C 1-3}) in Lemma (\ref{den 2}), but it is easy to see it may fail in some cases. We conclude this section with two such examples, and for ease in writing we will refer to the continuity posed for (\ref{C 1-3}) as condition C2.

The first example where condition C2 fails is for
$$ 
X_t = \max \{0,B_t\}, t \in [0,1], 
$$
where $\{B(t): t \in [0,1]\}$ is a sample continuous Brownian motion such that $P(B(0)=0)=1.$
Thus for $h_1(t)=0, t \in [0,1],$ and $h_2(t)= \delta>0, t \in [0,1],$ we have
$$
P(h_1\preceq X)-P(h_2 \preceq X)=1,
$$
no matter how small the constant $\delta$. If we let
$$
X_t = Z + \max\{0,B_t\}, 
$$
where $Z$ is independent of the Brownian motion $B$, then  for $h_1(t)=c, t \in [0,1],$ and $h_2(t)= c+\delta>0, t \in [0,1],$  an easy calculation implies
$$
P(h_1\preceq X)-P(h_2 \preceq X)= P(c\le Z <c+\delta).
$$
Hence, if $P(Z=c)>0$, then again no matter how small the constant $\delta$, the condition C2 fails. If $Z$ has a continuous distribution, then C2 holds for all choices of $h_1,h_2 \in C[0,1]$ constant functions, but that it actually does or does not satisfy C2 is not at all obvious.

The second example fails C2 for the same reasons as those in the previous one, i.e. wherever the process starts, it never goes below that level, and if it starts at a fixed point which has positive probability, then C2 fails. However, it differs in that its paths are Lip-1 with probability one. The example is
$$ 
X_t = \int_0^t N(s)ds, t \in [0,1], 
$$
where $N=\{N(t): t \in [0,1]\}$ is a Poisson process with parameter one and cadlag paths. Properties such as those mentioned for example one also hold here. The details are left to the reader.

\section{Additional Asymptotics for Half-Region Depth}

Our next result shows the consistency result of (\ref{LLN}) can be refined to include rates of convergence provided we restrict the set $E$ to be a sup-norm compact subset of $M(T)$ satisfying
the entropy condition 
\begin{align}\label{ent}
\int_{0^+} (\log N(E, \epsilon, ||\cdot||_{\infty}))^{\frac{1}{2}}\epsilon^{-\frac{1}{2}} d\epsilon < \infty, 
\end{align}
where $N(E, \epsilon, ||\cdot||_{\infty}))$ is the covering number of $E$ with $\epsilon$-balls in the $||\cdot||_{\infty}$-norm. 
In particular, since the processes
\begin{align}\label{CLT}
\{\sqrt{n} (D_n(h) - D(h)): h \in E\}, n \ge 1, 
\end{align}
live in $\ell_{\infty}(E)$, we examine their asymptotic behavior in that setting, and in Corollary \ref{rates} produce sub-Gaussian tail bounds that are uniform in $n$.The basic notation is as in section 3, and we freely use the empirical process ideas for weak convergence in the space $\ell_{\infty}(E)$ as presented in \cite{Dudley-unif-clt} and \cite{vw}.
\bigskip

In the proof of these results we have need for the stochastic processes $\{H_{n,1,h}: h \in E\}, n \ge 1,$ and $\{H_{n,2,h}: h \in E\}, n \ge 1,$ where
\begin{align}\label{CLT2}
H_{n,1,h} \equiv \frac{1}{\sqrt n} \sum_{j=1}^n[ I( X_j \succeq h)- P( X_j \succeq h)], 
\end{align}
and
\begin{align}\label{CLT3}
H_{n,2,h} \equiv \frac{1}{\sqrt n} \sum_{j=1}^n [I( X_j \preceq h)\}-P( X_j \preceq h)].
\end{align}
The first step of our proof will be to show that each of these processes satisfies the CLT in $\ell_{\infty}(E)$
with limits that are centered, sample path bounded, Gaussian processes $G_1=\{G_{1,h}: h \in E\}$ and
$G_2=\{G_{2,h}: h \in E\}$, respectively, that are uniformly continuous on $E$ with respect to their $L_2$-distances, and have covariance functions
\begin{align}\label{CLT4}
E(G_{1,h_1}G_{1,h_2}) = P(X\succeq h_1, X \succeq h_2)- P(X\succeq h_1)P( X \succeq h_2), h_1,h_2 \in E, 
\end{align}
and 
\begin{align}\label{CLT5}
E(G_{2,h_1}G_{2,h_2}) = P(X\preceq h_1, X \preceq h_2)- P(X\preceq h_1)P( X \preceq h_2), h_1,h_2 \in E.
\end{align}
In the following theorem these Gaussian processes also appear in connection with the limiting finite dimensional distributions of the centered empirical half-region depth processes given in (\ref{CLT}), see (\ref{CLT7}-\ref{CLT9}).

\begin{thm}\label{bdd prob clt}
Let $X(t) = Y(t) + Z, t \in T,$ where $Y= \{Y(t): t \in T\}$ has sample paths in the linear space $M(T)$ 
and $Z$ is independent of $Y$ with density $f_Z(\cdot)$  on $\mathbb{R}$ that is absolutely continuous and its derivative $f_Z^{'}(\cdot)$ is  in $L_1(\mathbb{R})$. Also, assume $X_1,X_2, \cdots$ are i.i.d. copies of the process $X$ with sample paths in $M(T)$ and that $X,X_1,X_2,\cdots$ are defined on the probability space  $(\Omega,\mathcal{F},P)$.
If $E$ is a sup-norm compact subset of $M(T)$ satisfying the entropy condition (\ref{ent}), then
\begin{align}\label{CLT6}
\lim_{r \rightarrow \infty} \sup_{n \ge 1}P^{*}( \sup_{h \in E}\sqrt{n} |D_n(h) - D(h)|  \ge r) =0,  
\end{align} 
where $P^{*}$ denotes the outer probability for subsets of $(\Omega,\mathcal{F},P)$. 
Furthermore, there is a stochastic process $\{\Gamma_h: h \in E \}$ 
such that the finite dimensional distributions of the processes $\{\sqrt n (D_n(h)-D(h)): h \in E\}, n\ge 1$ converge weakly to  $\{\Gamma_h: h \in E\} $, where
\begin{align}\label{CLT7}
\mathcal{L}(\Gamma_h) =\mathcal{L}(G_{1,h}){\rm ~{for}~}  h \in E {\rm ~{and}~} P(X\succeq h)< P(X \preceq h),
\end{align} 
\begin{align}\label{CLT8}
 \mathcal{L}(\Gamma_h) =\mathcal{L}(G_{2,h}){\rm~{for}~} h \in E {\rm~ {and}~}  P(X\preceq h)< P(X \succeq h), 
\end{align} 
and
\begin{align}\label{CLT9}
\mathcal{L}(\Gamma_h) =\mathcal{L}(\min\{G_{1,h},G_{2,h}\}) {\rm~{for}~} h \in E {\rm~ {and}~} P(X\succeq h)=P(X \preceq h).
\end{align} 
\end{thm}
\bigskip

\begin{proof} Since (\ref{min-min}) holds and $X,X_1,X_2,\cdots $ are i.i.d. we have 
\begin{align}\label{CLT10}
\sqrt n|D_n(h) - D(h)| \le |H_{n,1,h}| + |H_{n,2,h}|.
\end{align}
Hence (\ref{CLT6}) will hold provided we show 
\begin{align}\label{CLT11}
\lim_{r \rightarrow \infty} \sup_{n \ge 1}P^{*}( \sup_{h \in E}|H_{n,1,h}|   \ge r) =0,  
\end{align}
and 
\begin{align}\label{CLT12}
\lim_{r \rightarrow \infty} \sup_{n \ge 1}P^{*}( \sup_{h \in E}|H_{n,2,h}|   \ge r) =0.  
\end{align}

To verify (\ref{CLT11}) and (\ref{CLT12}) it suffices to show that the stochastic processes $\{H_{n,1,h}: h \in E\}$ and  $\{H_{n,2,h}: h \in E\}$ converge weakly in $\ell_{\infty}(E)$ to the centered Gaussian processes $G_1$ and $G_2$, respectively. That is, once these CLT's hold, then item (iii) of Theorem 1.3.4 of \cite{vw} provides the conclusion we need.

In order to formulate these CLTs in $\ell_{\infty}(E)$ we let $\mathcal{C}$ be a family of subsets of $M(T)$ indexed by $E$, where
\begin{align}\label{CLT13}
\mathcal{C}= \mathcal{C}_{\inf} \cup \mathcal{C}_{\sup},
\end{align}
\begin{align}\label{CLT14}
\mathcal{C}_{\inf}= \{C_h: h \in E\}~\rm {and}~\mathcal{C}_{\sup}= \{\hat{C}_h: h \in E\}, 
\end{align}
\begin{align}\label{CLT15}
C_h=\{ z \in M(T): \inf_{t \in T}(z(t) -h(t)) \ge 0\}, 
\end{align}
and 
\begin{align}\label{CLT16}
\hat{C}_h=\{ z \in M(T): \sup_{t \in T}(z(t) -h(t)) \le 0\}, 
\end{align}
Of course, since we are assuming $M(T)$ is a linear space such that (\ref{Tcond}) holds we have the inf and sup defining the sets $C_h$ and $D_h$, respectively, are the same when $t \in T$ is replaced by $t \in T_0$.

Since $\ell_{\infty}(E)$ is a separable Banach space only when $E$ is finite, we need to use weak convergence in the non-separable setting, and proceed to verify that $\mathcal{C}_{\inf}$ and 
$\mathcal{C}_{\sup}$ are both $P$-Donsker classes of sets. Then, since a finite union of $P$-Donsker classes is  $P$-Donsker, we will have $\mathcal{C}$ also $P$-Donsker.

To show $\mathcal{C}_{\inf}$ is $P$-Donsker we recall the stochastic process $\{W_h: h  \in M(T)\}$ on $(\Omega,\mathcal{F},P)$ given in (\ref{Wdef}). Then, the path $X(t,\cdot)$ is in $C_h$ if and only if $W_h(\cdot) \ge 0$, and we also have $X \succeq h$ on $T$ if and only if $W_h \ge 0$. Therefore, $\mathcal{C}_{\inf}$ $P$-Donsker will imply that the empirical processes $\{H_{n,1,h}: h \in E\}$  as given in (\ref{CLT2}) converge in distribution on $\ell_{\infty}(E)$ to a centered Gaussian measure $\gamma_{\inf}$ with separable support in $\ell_{\infty}(E)$. Furthermore, $\gamma_{\inf}$ is induced by the Gaussian process $G_1$ as indicated above.

Since (\ref{ent}) holds, for every $\delta>0$ there are $N_{\delta} \equiv N(E, \delta, ||\cdot||_{\infty})$ functions $h_1,\cdots,h_{N_{\delta}}$ in $E$ such that the brackets
\begin{align}\label{brack 1}
F(\delta,h_j)=\{ z \in M(T): h_j(t)-\delta \le z(t)\le h_j(t) + \delta\}, j=1,\cdots,N_{\delta}, 
\end{align} 
have union covering $E$. Furthermore,  $ z \in F(\delta,h_j)$ implies
$$
I(X \succeq h_j +\delta) \le I(X\succeq  z) \le I(X\succeq h_j -\delta).
$$
Hence, for $\delta>0$ fixed we have from (\ref{W 3}) that
$$
\mathcal{E} \equiv \{I( X\succeq h): h \in E\} 
$$
is a subset of
$$
\cup_{j=1}^{N_{\delta}} \{I(X \succeq z): z \in E, I(X\succeq h_j+\delta) \le I(X \succeq z) \le I(X \succeq h_j-\delta)  \}.
$$
Furthermore,
$$
||I(X \succeq h_j -\delta)- I(X \succeq h_j+\delta) ||_{2}^2 = ||I(X \succeq h_j -\delta)- I(X \succeq h_j+\delta) ||_{1} 
$$
$$
~~~~~~~~~~~~~~~~~~~~~~~~~~~~= P(X \succeq h_j-\delta) - P(X_j \succeq h_j +\delta) \le 4\delta\int_{\mathbb{R}}|f_Z^{'}(x)|dx,
$$
where $||\cdot||_p$ denotes the $L_p$  norm with respect to $P$, and the inequality follows from (\ref{W 3}) and (\ref{brack 1}).
Now
$$
\int_{0^+}(\log N(\mathcal{E}, x, ||\cdot||_2))^{\frac{1}{2}}dx= \int_{0^+}(\log N(\mathcal{E}, x^2, ||\cdot||_1))^{\frac{1}{2}}dx
$$
$$
\le \int_{0^+}(\log N(\mathcal{E}, x^2, ||\cdot||_{\infty}))^{\frac{1}{2}}dx
$$
where the inequality follows since $||\cdot||_1 \le ||\cdot||_{\infty}$ on $M(T)$. Letting $s=x^2$ in the right most integral above and applying (\ref{ent}) we have
\begin{align}\label{brack 2}
\int_{0^+}(\log N(\mathcal{E}, x, ||\cdot||_2))^{\frac{1}{2}}dx \le  \int_{0^+}(\log N(\mathcal{E}, s, ||\cdot||_{\infty}))^{\frac{1}{2}}s^{-\frac{1}{2}}ds < \infty. 
\end{align}

Hence by Ossiander's CLT with bracketing \cite{ossiander-bracketing}, or as in \cite{Dudley-unif-clt}, p 239, we have $\mathcal{E}$  a $P$-Donsker class of functions, which implies $\mathcal{C}_{\inf}$ is a $P$-Donsker class of sets. Hence the empirical processes $\{H_{n,1,h}: h \in E\}$ given in (\ref{CLT2}) converge weakly in $\ell_{\infty}(E)$ to the centered Gaussian process $G_1$ induced by the Radon Gaussian measure $\gamma_{\inf}$ and has covariance as indicated in (4.5). A similar result holds for the empirical processes $\{H_{n,2,h}: h \in E\}$ given in (\ref{CLT3}), which therefore satisfy the CLT in $\ell_{\infty}(E)$ with centered Gaussian limit $G_2$. Hence (\ref{CLT6}) is proven.
\bigskip

The next step of our proof is to show the finite dimensional distributions of the stochastic processes in (\ref{CLT})
converge. To check this we set
$$
F_n(h) = \frac{1}{ n} \sum_{j=1}^n I( X_j \preceq h),~~~ F(h)=P( X \preceq h), 
$$
and 
$$
G_n(h)= \frac{1}{n} \sum_{j=1}^n I( X_j \succeq h),~~~ G(h)= P( X \succeq h). 
$$

Hence, let $I=I_1 \cup I_2 \cup I_3$, where $I_1,I_2,I_3$ are disjoint,
$$
I_1= \{h_1,\cdots,h_{r_1}\},~ I_2=\{h_{r_1+1},\cdots,h_{r_{2}}\},~I_3=\{h_{r_2+1},\cdots,h_r\},
$$
and
$$
I_1= \{ h \in I: F(h)<G(h)\},
$$
$$
I_2= \{ h \in I: F(h)>G(h)\},
$$
and
$$
I_3= \{ h \in I: F(h)=G(h)\}.
$$
Setting
$$
V_n(h)= \sqrt n(D_n(h)-D(h)),
$$
we have
$$
V_n(h)= \sqrt n [\min(F_n(h),G_n(h)) - \min(F(h),G(h))],
$$
and since $I$ is an arbitrary subset of $E$ to prove the finite dimensional distributions of the processes in (\ref{CLT}) we need to show
$$
(V_n(h_1),\cdots,V_n(h_r)), 
$$
converges in distribution on $\mathbb{R}^r$.

For $n\ge 1$ let
$$
U_n(h)= \sqrt n (F_n(h)- F(h), h \in I_1,
$$
$$
U_n(h)=\sqrt n(G_n(h) - G(h)), h \in I_2
$$
$$
U_n(h)=\sqrt n \min (F_n(h)-F(h), G_n(h)-G(h)), h \in I_3,
$$
and take
$$
N(\omega)= \min\{m\ge 1: U_n(h_i)=V_n(h_i), 1 \le i \le r_2,~n\ge m\}.
$$
Then, the strong law of large numbers implies $P(N<\infty)=1$, and $U_n(h)=V_n(h)$ for all $h \in I$ and all $n \ge N$. Therefore,
$$
\lim_{n \rightarrow \infty}P(\sup_{h \in I}|U_n(h)-V_n(h)| \ge \epsilon) \le \lim_{n \rightarrow \infty}P(N >n)=0,
$$
and the convergence of the finite dimensional distributions will hold if we show 
$$
T_n= u_1U_n(h_1)+\cdots+u_rU_n(h_r) 
$$
converges in distribution for all $(u_1,\cdots,u_r) \in \mathbb{R}^r$. Setting
$$
S_n= \sum_{j=1}^{r_1} u_j(F_n(h_j)-F(h_j)) +  \sum_{j=r_1 +1}^{r_2} u_j(G_n(h_j)-G(h_j)),
$$
we have
$$
T_n= \sqrt n\big [\min [S_n+u_{r_{2}+1}[F_n (h_{r_{2} +1}) -F(h_{r_{2} +1})], S_n+u_{r_{2}+1}[G_n (h_{r_{2} +1}) -G(h_{r_{2} +1})] \big ] +
$$
$$
 \sqrt n \sum_{j=r_{2} +2}^r u_j \min [F_n (h_{j}) -F(h_{j}), G_n (h_{j}) -G(h_{j})]. 
$$
If
$$
\Lambda(a_1,b_1,a_2,b_2,\cdots,a_k,b_k) = \sum_{i=1}^k \min[a_i,b_i],
$$
then $\Lambda$ is continuous from $\mathbb{R}^{2k}$ to $\mathbb{R}^{k}$. Therefore, if $k=r-r_2$ with
$$
R_n= (a_{n,1},b_{n,1},\cdots,a_{n,r-r_{2}},b_{n,r-r_{2}})
$$
and
$$
a_{n,1}=\sqrt n \big(S_n+u_{r_{2}+1}[F_n (h_{r_{2} +1}) -F(h_{r_{2} +1})]\big),
$$
$$
b_{n,1}= \sqrt n \big (S_n+u_{r_{2}+1}[G_n (h_{r_{2} +1}) -G(h_{r_{2} +1})]\big), 
$$
$$
a_{n,i}=\sqrt n u_{r_{2}+i}[F_n (h_{r_{2} +i}) -F(h_{r_{2} +i})], i=2,\cdots, r-r_2,
$$
$$
b_{n,i}=\sqrt n u_{r_{2}+i}[G_n (h_{r_{2} +i}) -G(h_{r_{2} +i})], i=2,\cdots,r-r_2,
$$
we have $R_n$ converging weakly to a centered Gaussian random variable, i.e. it is a  sum of independent vectors in $\mathbb{R}^{2(r-r_{2})}$ whose summands are indicator functions multiplied by $u_j$'s. Now
$\Lambda(R_n)=T_n,$ 
and thus the continuous mapping theorem implies $T_n$ converges in distribution. Since the vector $(u_1,\cdots,u_r) \in \mathbb{R}^r$ is arbitrary, the finite dimensional distributions converge. Of course, the claims in (\ref{CLT7}), (\ref{CLT8}), and (\ref{CLT9})  involving the one dimensional distributions are also now proven.
\end{proof}

Next we turn to a corollary of Theorem \ref{bdd prob clt}, which provides sub-Gaussian tail bounds for the convergence to zero in (\ref{CLT6}). To avoid measurability issues arising in its proof, we assume the set $E$ is countable. Of course, under the assumption (\ref{Tcond}) and that $M(T)$ is a linear space, we have that the random vectors (stochastic processes) 
$$
\{D_n(h))-D(h): h \in E\} {\rm{~and~}} H_{n,i}:=\{H_{n,i,h}: h \in E\}, i =1,2,
$$
given in (\ref{CLT2}) and (\ref{CLT3}), take values in the Banach space $\ell_{\infty}(E)$ with norm $||x||_{\infty} = \sup_{h \in E}|x_h|$ for $x=\{x_h\} \in \ell_{\infty}(E)$. Hence, the assumption $E$ is countable implies these random vectors on $(\Omega, \mathcal{F},P)$ are $\ell_{\infty}(E)$ valued  in the sense used in \cite{led-tal-book}, so for the convenience of the reader we freely quote from this single source a number of results used in the proof. However, from a historical point of view it should be observed that an important first step in these results involves the Hoffmann-J\o rgenesen inequalities obtained in \cite{hoff}, and for series and a.s. normalized partial sums of sequences of independent random vectors, some results of a similar nature appeared in \cite{jain-marcus-integ-75} and \cite{kuelbs-expon-moments-78}.

\begin{notat} Let $X$ take values in a Banach space $B$ with norm $||x|| = \sup_{f \in D}|f(x)|, x  \in B,$ where $D$ is a countable subset of the unit ball of the dual space of $B$ and  $f(X)$ is measurable for each $f \in D$. Then, we are in the setting used in Chapter 6 of \cite{led-tal-book}, and the $\psi_2$-Orlicz norm of $||X||$ is given by
$$
||X||_{\psi_2} = \inf\{c>0: E( \exp\{(\frac{||X||}{c})^2\}) \le 2\}. 
$$
\end{notat}

\begin{cor}\label{rates} 
Let $H_{n,i}:=\{ H_{n,i,h}: h \in E\}, i =1,2$, be the stochastic processes in  (\ref{CLT2}) and (\ref{CLT3}), and for $i=1,2$
\begin{align}\label{CLT30}
||H_{n,i}||_{\infty} = \sup_{h \in E}|H_{n,i,h}|.
\end{align}
Then, under the assumptions of Theorem \ref{bdd prob clt} and  that $E$ is countable, we have  
\begin{align}\label{CLT31}
\hat k =\sup_{n\ge 1,i=1,2}E(||H_{n,i}||_{\infty})< \infty, 
\end{align}
and there exists an absolute constant $k_2< \infty$ 
such that for any $r>0$ 
\begin{align}\label{CLT32}
\sup_{n \ge 1}P(\sup_{h \in E} \sqrt n |D_n(h)- D(h)| \ge r) \le 4\exp\{-\alpha r^2\} 
\end{align}
provided $\alpha>0$ is sufficiently small that
\begin{align}\label{CLT33}
\sqrt{4\alpha} k_2(\hat k+2) <1. 
\end{align}
\end{cor}

\begin{proof} From (\ref{CLT10}) 
$$
P(\sup_{h \in E} \sqrt n |D_n(h)- D(h)| \ge r) \le P((||H_{n,1}||_{\infty} \ge \frac{r}{2}) +  P((||H_{n,2}||_{\infty} \ge \frac{r}{2}).
$$
Hence, Markov's inequality implies
$$
P(\sup_{h \in E} \sqrt n |D_n(h)- D(h)| \ge r) \le \exp\{-\alpha r^2\} \sum_{i=1}^2E(\exp\{4\alpha||H_{n,i}||_{\infty}^2\}),
$$
and (\ref{CLT32}) holds provided $\alpha>0$ is sufficiently small that
\begin{align}\label{CLT34}
\sqrt{4\alpha}\sup_{n \ge 1, i=1,2} ||H_{n,i}||_{\infty, \psi_2} <1, 
\end{align}
where we write $||H_{n,i}||_{\infty, \psi_2}$ to denote the $\psi_2$-norm of $||H_{n,i}||_{\infty}$. Now Theorem 6.21 of \cite{led-tal-book} implies there exists an absolute constant $k_2<\infty$ such that
\begin{align}\label{CLT35}
||H_{n,i}||_{\infty,\psi_2} \le k_2[ E(||H_{n,i}||_{\infty}) + (\sum_{j=1}^n ||\frac{Y_j}{\sqrt n}||_{\infty,\psi_2}^2)^{\frac{1}{2}}], 
\end{align}
where $\{Y_j: j \ge 1\}$ are independent, mean zero, $\ell_{\infty}(E)$ valued random vectors with $ Y_j= \{I( X_j \succeq h)- P(X_j \succeq h): h \in E\}$ for $j \ge 1$ when $i=1$, and $ Y_j= \{I( X_j \preceq h)- P(X_j \preceq h): h \in E\}$ for $j \ge 1$ when $i=2$. Since $||Y_j||_{\infty} \le 1$, we have $ ||\frac{Y_j}{\sqrt n}||_{\infty,\psi_2} \le \frac{1}{ \sqrt{ n \log(2) } }$, and (\ref{CLT35}) implies
\begin{align}\label{CLT36}
||H_{n,i}||_{\infty,\psi_2} \le k_2[ E(||H_{n,i}||_{\infty}) + (\frac{1}{\log(2)})^{\frac{1}{2}}],
\end{align}

Now from (\ref{CLT34}) and (\ref{CLT36}) we have (\ref{CLT32}) for $\alpha>0$ sufficiently small that
\begin{align}\label{CLT37}
\sqrt{4 \alpha}k_2[ \hat k + (\frac{1}{\log(2)})^{\frac{1}{2}}] < 1, 
\end{align}
provided (\ref{CLT31}) holds.

Hence, to complete the proof we must prove $\hat k <\infty$. To accomplish this  we first show
\begin{align}\label{CLT38}
\sup_{n\ge 1}E(||H_{n,1}||_{\infty})< \infty. 
\end{align}
This follows from Proposition 6.8 of \cite{led-tal-book}  applied to the partial sums $S_k$ of the $\{Y_j:  j\ge 1 \}$ with $p=1$ provided we show $\sup_{n\ge 1} t_{0,n} <\infty$, where
\begin{align}\label{CLT39}
t_{0,n}=\inf\{t>0: P(\max_{1 \le k \le n}||\frac{S_k}{\sqrt n}||_{\infty}>t)  \le \frac{1}{8}\}, n \ge 1. 
\end{align}
Moreover, $E$ countable implies Ottaviani's inequality is available as in Lemma 6.2 of \cite{led-tal-book}], and hence for every $u,v>0$
\begin{align}\label{CLT40}
P(\max_{1 \le k \le n}||\frac{S_k}{\sqrt n}||_{\infty}>u+v) \le \frac{P(||\frac{S_n}{\sqrt n}||_{\infty}>v)}{ 1 - \max_{1 \le k \le n} P(||\frac{S_n-S_k}{\sqrt n}||_{\infty} >u)}. 
\end{align}

Furthermore, since $\frac{S_n}{\sqrt n} = H_{n,1}$ for $n \ge 1$, and the proof of Theorem 2 implies $\{H_{n,1}: n \ge 1\}$ satisfies the central limit theorem in $\ell_{\infty}(E)$, the Portmanteau Theorem (applied to closed sets) implies there exists $u_0< \infty$ such that for $u \ge u_0$ 
$$
P(||\frac{S_m}{\sqrt m}||_{\infty} \ge u) \le \frac{1}{2}
$$
for all $m \in [ m_0, \infty)$. Therefore, there exists $u_1\in [u_0,\infty)$ such that
$$
 \sup_{m \ge 1}P(||\frac{S_m}{\sqrt m}||_{\infty} \ge u_1) \le \frac{1}{2},
$$ 
and hence
$$
\sup_{n \ge 1} \max_{1 \le k \le n} P(||\frac{S_n-S_k}{\sqrt n}||_{\infty} \ge u_1) \le  \sup_{n \ge 1}\max_{1 \le m  \le n}P(||\frac{S_m}{\sqrt m}||_{\infty} \ge u_1)  \le \frac{1}{2}.
$$ 
Thus
(\ref{CLT40}) implies for all $v>0$ and $n \ge 1$ that
\begin{align}\label{CLT41}
P(\max_{1 \le k \le n}||\frac{S_k}{\sqrt n}||_{\infty}>u_1+v) \le 2P(||\frac{S_n}{\sqrt n}||_{\infty}>v). 
\end{align}
Again, by the central limit theorem there exists there exists $v_1< \infty$ such that $v \ge v_1$ implies
$$
2 \sup_{n \ge 1}P(||\frac{S_n}{\sqrt n}||_{\infty}\ge v) \le \frac{1}{8}, 
$$
and hence we  see from (\ref{CLT41}) that $\sup_{n\ge 1} t_{0,n} \le u_1+v_1 <\infty$ when $i=1$ (and the partial sums come from the $\{Y_j: j \ge 1\}$). However, the same proof applies when $i=2$ and the partial sums are formed from  $\{Z_j: j \ge 1\}$, where $Z_j= \{I( X_j \preceq h)- P(X_j \preceq h): h \in E\}$ for $j \ge 1$. Hence the proof is complete.
\end{proof}

\section{Half-Region Depth over Finite Subsets}

In order to make half-region depth more amenable to discrete computations we now define half-region depth over finite sets, and prove a uniform consistency result in this setting.  

As before we assume $X:=\{X(t)=X_{t}\colon t\in T\}$ is a stochastic process on the probability space $(\Omega, \mathcal{F},P)$, all of whose sample paths are in $M(T)$.

If $h \in M(T)$ we define the half-region $P$-depth of $h$ with respect to  $J \subseteq T$ to be
\begin{align}\label{JHR}
D_J(h)=\min\{P(X \succeq_J h), P(X\preceq_J h)\}, 
\end{align}
where $h_1 \succeq_J h_2$ ($h_1 \preceq_J h_2$)  holds for functions $h_1,h_2$ defined on $T$ if $h_1(t) \ge h_2(t)$  ($h_1(t) \le h_2(t)$ for all $t \in J$.

Let $X_1,X_2, \cdots$ be i.i.d. copies of the process $X$, and assume $X,X_1,X_2,\cdots$ are defined on $(\Omega,\mathcal{F},P)$ suitably enlarged, if necessary, and that all sample paths of each $X_j$ are in $M(T)$. Then, the empirical half-region
depth of $h \in M(T)$ over a set $J\subseteq T$ is given by
\begin{align}\label{JEHR}
D_{n,J}(h)=\min\{\frac{1}{n} \sum_{j=1}^n I( X_j \succeq_J h), \frac{1}{n} \sum_{j=1}^n I( X_j \preceq_J h)\},
\end{align}

For $h \in M(T)$  and $J$ any finite subset of $T$,  the probabilities in (\ref{JHR}) are defined, and the events in (\ref{JEHR}) are in $\mathcal{F}$. Therefore, the classical law of large numbers implies with probability one
\begin{align}\label{JLLN}
\lim_{n \rightarrow \infty} |D_{n,J}(h)-D_J(h)|=0. 
\end{align}

The next theorem refines (\ref{JLLN}) to be uniform over $h$ and $J$, as long as $J \in \mathcal{J}_r$, where for each integer $r \geq 1$, 
$$
\mathcal{J}_r=\{ J \subseteq T: \#J \leq r\},
$$
and $\# J$ denotes the cardinality of the set $J$.

\begin{thm}\label{J consist}
Let $X,X_1,X_2,\cdots$ be i.i.d. copies of the stochastic process $X= \{X(t): t \in T\}$ defined on the probability space  $(\Omega,\mathcal{F},P)$, and all of whose sample paths
are in the linear space M(T). Let 
\begin{align}\label{C-sets}
\mathcal{C}=\{C_{t,y}: t \in T, y \in \mathbb{R}\},
\end{align}
where $C_{t,y}=\{z \in M(T): z(t) \le y\}$, and assume the empirical CLT holds with respect to the probability $ \mathcal{L}(X)$ over $\mathcal{C}$. Then, for every integer $r \ge 1$ fixed we have with probability one that
\begin{align}\label{JULLN}
\lim_{n \rightarrow \infty} [\sup_{h \in M(T)}\sup_{J \in \mathcal{J}_r} |D_{n,J}(h)-D_J(h)|]^{*}=0.
\end{align}
\end{thm}
 \bigskip

\begin{rem}
The implication in (\ref{JULLN}) does not follow from the corresponding finite dimensional result for half-region depth since the finite set $J$ is not fixed, but it is also the case that the assumption of an empirical CLT over $\mathcal{C}$ is non-trivial. Fortunately
 \cite{kkz} and \cite{kz-quant} provide many examples of processes that satisfy this empirical CLT, and to which Theorem \ref{J consist} applies. These include  a broad collection of Gaussian processes, compound Poisson processes, stationary independent increment stable processes, and martingales. 
Moreover, if  $J: \Theta \rightarrow \mathcal{J}_r$, then
$$
[\sup_{h \in M(T)}\sup_{\theta \in \Theta}|D_{n,J(\theta)}(h) - D_{J(\theta)}(h)|]^{*} \le[\sup_{h \in M(T)} \sup_{J \in \mathcal{J}_r}|D_{n,J}(h) - D_{J}(h)|]^{*},
$$
and hence it is immediate that (\ref{JULLN}) holds when the choice of $J$ is  arbitrarily parameterized by $\Theta$ as long as $\# J(\theta) \le r, \theta \in \Theta$. 
\end{rem}

\begin{proof} 
Since (\ref{min-min}) holds, we have
$$
|D_{n,J}(h)-D_J(h)| \le  A_n(J,h) +  B_n(J,h), 
$$
where
$$
A_n(J,h)=|\frac{1}{n}\sum_{j=1}^n I( X_j \succeq_J h) -P(X\succeq_J h)|  
$$
and
$$
B_n(J,h)=|\frac{1}{n}\sum_{j=1}^n I( X_j \preceq_J h) -P(X\preceq_J h)|.
$$
Therefore, sub-additivity of measurable cover functions implies (\ref{JULLN}) will follow once we verify that with probability one
$$
\lim_{n \rightarrow \infty} [ \sup_{h \in M(T)} \sup_{J \in \mathcal{J}_r} A_n(J,h)]^{*} = \lim_{n \rightarrow \infty} [ \sup_{h \in C[0,1]} \sup_{J \in \mathcal{J}_r} B_n(J,h)]^{*} =0. 
$$

Fix an integer $r \ge 1,$ and set 
$$
\phi(u_1,\cdots,u_r)= \min\{u_1,\cdots,u_r\}.
$$
Then, for $J \in \mathcal{J}_r$, $h \in M(T)$, and $j=1,\cdots,r$, define
$$
f_j= I_{C_{t_j,h(t_j)}},
$$
which implies
$$
D_{J,h}= \{z \in M(T): z(t) \le h(t), t \in J\} = \cap_{j=1}^r C_{t_j, h(t_j)},
$$
and
$$
I_{D_{J,h}} = \phi(f_1,\cdots,f_r).
$$
Since $\mathcal{C}$ is a Donsker class with respect to $P$ and $r\ge 1$ is fixed, Theorem 2.10.6 of \cite{vw} implies
$$
\mathcal{D}= \{ D_{J,h}: J \in \mathcal{J}_r, h \in M(T)\}
$$
is also $P$-Donsker with respect to $P$. Thus by Lemma 2.10.14 of \cite{vw} we have almost surely that
\begin{align}\label{JULLNB}
\lim_{n \rightarrow \infty} [ \sup_{h \in M(T)} \sup_{J \in \mathcal{J}_r} B_n(J,h)]^{*} =0.  
\end{align}

Since $\tilde{ \mathcal{D}} =\{ D_{J,h}^c: D_{J,h} \in \mathcal{J}_r\}$ is then also a Donsker class, the above argument implies that
\begin{align}\label{JULLNA}
\lim_{n \rightarrow \infty} [ \sup_{h \in M(T)} \sup_{J \in \mathcal{J}_r} A_n(J,h)]^{*} =0. 
\end{align}
Combining (\ref{JULLNB}) and (\ref{JULLNA}) implies (\ref{JULLN}).
\end{proof}

\end{document}